\documentclass[12pt,a4paper]{amsart}
\usepackage[all]{xy} 		
\usepackage{longtable}
\usepackage{tikz}
\usepackage{amsxtra,amsmath,amsfonts,amscd, amssymb, mathrsfs}
\usepackage{fullpage}
\usepackage[colorlinks=false,hyperindex]{hyperref}

\hypersetup{
	colorlinks,
	citecolor=red,
	linkcolor=blue,
	urlcolor=blue}

\usepackage{color}

\makeatletter
\@namedef{subjclassname@2010}{%
 \textup{2010} Mathematics Subject Classification}
\makeatother

\newtheorem{thm}{Theorem}[section]

\newtheorem{lem}[thm]{Lemma}
\newtheorem{conj}[thm]{Conjecture}

\theoremstyle{definition}
\newtheorem{defn}[thm]{Definition}
\newtheorem{example}[thm]{Example}
\theoremstyle{remark}
\newtheorem{rem}[thm]{Remark}
\numberwithin{equation}{section}
\title{Deligne--Illusie Classes as Arithmetic Kodaira--Spencer Classes}

\author{Taylor Dupuy}%ducky
\address{Taylor Dupuy, Department of Mathematics \& Statistics, University of Vermont}
\email{taylor.dupuy@uvm.edu}
\author{David Zureick-Brown}
\address{David Zureick-Brown, Dept. of Math and CS, Emory University, 400 Dowman
 Dr., W401, Atlanta, GA 30322, USA}
\email{dzb@mathcs.emory.edu}

\newcommand{\defi}[1]{\textsf{#1}} 				% for defined terms

 \newcommand{\Qbar}{\overline{\Q}}
\def\presuper#1#2%
  {\mathop{}%
   \mathopen{\vphantom{#2}}^{#1}%
   \kern-\scriptspace%
   #2}

\newcommand{\GSp}{\operatorname{GSp}}
\newcommand{\DI}{\operatorname{DI}}
\newcommand{\KS}{\operatorname{KS}}

\newcommand{\OO}{\mathcal{O}}
\newcommand{\Def}{\operatorname{Def}}
\newcommand{\id}{\operatorname{id}}
\newcommand{\FDer}{\operatorname{FDer}}

\newcommand{\CRing}{\mathsf{CRing}}
\newcommand{\Spec}{\operatorname{Spec}}
\newcommand{\eps}{\varepsilon}
\newcommand{\Der}{\operatorname{Der}}
\newcommand{\pDer}{\pi\operatorname{-Der}}
\newcommand{\pr}{\operatorname{pr}}
\newcommand{\Z}{\mathbf{Z}}
\newcommand{\Fbar}{\overline{\mathbf{F}}}

\newcommand{\Jac}{\mathrm{Jac}}
\newcommand{\smCM}{\mathrm{smCM}}
\newcommand{\C}{\mathbf{C}}
\newcommand{\Q}{\mathbf{Q}}
\newcommand{\End}{\operatorname{End}}
\newcommand{\DefAS}{\operatorname{DefAS}}
\newcommand{\G}{\mathbf{G}}
\newcommand{\Hom}{\operatorname{Hom}}
\newcommand{\mm}{\frak{m}}
\newcommand{\qCL}{\operatorname{qCL}}

\newcommand{\FT}{FT}
\newcommand{\NN}{\mathbf{N}}
\newcommand{\Aut}{\operatorname{Aut}}
\newcommand{\ZZ}{\mathbf{Z}}
\newcommand{\Spf}{\operatorname{Spf}}
\renewcommand{\AA}{\mathbf{A}}
\newcommand{\Ext}{\operatorname{Ext}}
\newcommand{\FF}{\mathbf{F}}

\newcommand{\AL}{\operatorname{AL}}
\newcommand{\Autu}{\underline{\operatorname{Aut}}}
\newcommand{\pp}{\frak{p}}

\newcommand{\pfrak}{\frak{p}}

\begin{document}

%\tableofcontents
%\newpage

\maketitle

\begin{abstract}
	Faltings showed that ``arithmetic Kodaira--Spencer classes'' satisfying a certain compatibility axiom cannot exist.
	By modifying his definitions slightly, we show that the Deligne--Illusie classes satisfy what could be considered an ``arithmetic Kodaira--Spencer'' compatibility condition. 
	
	Afterwards we discuss a ``wittfinitesimal Torelli problem'' and its relation to CM Jacobians.
\end{abstract}

\section{Introduction}
The abstract of the paper ``Does there exist an Arithmetic Kodaira--Spencer class?" \cite{Faltings1999} is the following: ``We show that an analog of the Kodaira--Spencer class for curves over number-fields cannot exist.'' 
In the present paper we show that if we modify the axioms in \cite{Faltings1999} slightly such classes can exist; motivated by work of Buium and by work of Mochizuki, we give a candidate for such a class and discuss an application.
\begin{rem}
	The term ``arithmetic Kodaira--Spencer class'' is vague and the definition varies from paper to paper.
	In this paper we use 
	%two notions, 
	the Deligne--Illusie class (see \S \ref{S:deligne-illusie})%and the Buium class (\S \ref{S:lifted})
	. 
	More distinct ``arithmetic Kodaira--Spencer theory'' can be found in \cite{Dupuy2014a}, \cite{Buium2005} and \cite[\S 1.4]{Mochizuki2002}. %Definition 8.50
\end{rem}

We recall the setup of \cite{Faltings1999}. For schemes $S$ and $X$ of finite type over a base $B$ and a smooth map of $B$-schemes $\pi\colon X \to S$, we have an exact sequence 
\begin{equation}\label{E:rel-tan-seq}
0 \to \pi^*(\Omega_{S/B}) \to \Omega_{X/B} \to \Omega_{X/S} \to 0
\end{equation}
giving rise to a class $\kappa(X) \in \Ext^1(\Omega_{X/S},\pi^*\Omega_{S/B}) = H^1(X,T_{X/S}\otimes \pi^*\Omega_{S/B})$ which \cite{Faltings1999} calls the Kodaira--Spencer class. This induces the Kodaira--Spencer map $\KS_{\pi}\colon T_{S/B} \to R^1\pi_*T_{X/S}$. 
Such classes are important for many diophantine reasons and we refer the reader to \cite{Faltings1999} for a discussion.

The problem observed in \cite{Faltings1999} (and elsewhere) is that if $S$ is the spectrum of the ring of integers of a number field then there are no derivations and hence the Kodaira--Spencer map doesn't make sense.\footnote{Actually, $\Omega_{\OO_K/\ZZ}$ exists and the annhilator is the different, which controls ramification. 
This means for all but finitely many primes its localization will be zero. The theory we give presently gives something for unramified primes. } Although no map can exist, it is (a priori) possible for extensions corresponding to \eqref{E:rel-tan-seq} to exist in a canonical way (they don't as Faltings observes). 
For such extension classes to be canonical  \cite{Faltings1999} posits
that for morphisms $f\colon X \to Y$ of smooth $S$-schemes,
``Kodaira--Spencer classes with values in $\omega$'' (where $\omega =
\Omega_{S/B}$) should satisfy
\begin{equation}\label{E:faltings-compat}
  f^*(\kappa(Y)) = df_*(\kappa(X)) \in H^1\left(X, f^* \left( T_{Y/S} \otimes \omega\right)\right).
\end{equation}
Although \cite{Faltings1999} shows no such classes may exist, we show (using Buium's ``wittferential algebra'' \cite{Buium2005}, which formalizes the analogy between Witt vectors and power series) that there exist classes $\DI_{X_1/S_1}(\delta) \in H^1(X_0,F^*T_{X_0})$ which we call ``Deligne--Illusie classes'', and which satisfy a condition similar to \eqref{E:faltings-compat}. 
Here subscripts $n$ denote a reduction modulo $p^{n+1}$ and the reciepient sheaf here is the Frobenius tangent sheaf, whose local sections are Frobenius semi-linear derivations. 
The name stems from their implicit use in \cite{Deligne1987}. We show the following.
\begin{thm}
	 For a morphism $f\colon X \to Y$ of smooth $p$-formal schemes over $S = \Spf \ZZ_p$ which is either smooth or a closed immersion we have 
	\begin{equation}\label{E:di-compatibility}
	f^*\DI_{Y_1/S_1}(\delta) = df_* \DI_{X_1/S_1}(\delta) \in H^1(X_0, F^*_{X_0}T_{X_0} ).
	\end{equation}
\end{thm}

In section  \ref{S:notation} we give the analogies with the Kodaira--Spencer map, and we prove \eqref{E:di-compatibility} in section \ref{S:proofs}. 
In a separate paper we study the vector bundles coming from these extensions \cite{Dupuy2017}. 

Given the compatibility \eqref{E:compatibility} one may investigate the information this compatibility gives us in terms of (say) a map between a curve and its Jacobian. This leads to some interesting problems. In section \ref{S:applications} we investigate the ``wittfinitesimal Torelli problem'', which is the analogue the Torelli problem in our setting. This problem is related to Coleman's Conjecture concerning the finiteness of the number of CM Jacobians for genus bigger than 8. 
Let $A$ be an abelian variety over $\C$ of dimension $g$. Recall that every abelian variety with sufficiently many complex multiplications (see Definition \ref{def:CM}) can be defined over a number field; we define the \defi{field of moduli} of $A$ to be the intersection of all number fields over which $A$ is defined. Furthermore, every abelian variety with sufficiently many CM's has potentially good reduction. In what follows we will let $\DI^m(X_1)$ denote the obstruction to lifting an $m$th power of the Frobenius modulo $\pi^2$, where $\pi$ is a uniformizer in some finite extension of a full ring of $p$-typical witt vectors over a subfield of an algebraic closure of the field with $p$ elements.

\begin{lem}
\label{L:CM-field}
Let $C/\C$ be a pointed curve of genus $g$ and let $j\colon C \to A$ be its Abel--Jacobi map. Suppose that  $A$ is simple and let  $\Theta$ be the corresponding principal polarization on $A$. 
Fix a prime $p \neq 2$.  Then there exists some natural number $m = m(d,p,g)$ such that  
 $$dj_* \DI^m(C_1) \neq 0$$ 
implies that $A$ does not have a principally polarized CM structure $(A,\Theta,j)$ whose field of moduli is of degree less than $d$ over $\Q$.
\end{lem}
This proof of this proceeds by considering how lifts of the $q$-Frobenius are related to complex multiplication (\S\ref{S:applications}).

\section*{Acknowledgements}
We would like to thank Piotr Achinger, Jeff Achter, Alexandru Buium, Lance Gurney, Ehud Hrushovski, Eric Katz, Joe Rabinoff, Damian R\"{o}ssler, Ehud de Shalit, Ari Shnidman, Dinesh Thakur, and Christelle Vincent, for enlightening conversations. We especially thank Piotr Achinger for pointing out an alternative approach to this proof via \cite[Lemma 3.3.3]{achinger2017liftability}. 
This paper started during the first author's visit to the 2014 Spring Semester at MSRI on Model Theory. The first author was supported by the European Reseach Council under the European Union’s Seventh Framework Programme (FP7/2007-2013)/ ERC Grant agreement no. 291111/MODAG and the second author was supported by NSF CAREER Award (DMS-1555048).

\section{Notation and analogies}
\label{S:notation}

\subsection{Classical derivations/differentiation, and $\pi$-derivations/wittferentiation}
Let $\CRing$ denote the category of commutative rings. 
For $R \in \CRing$ we let $\CRing_R$ denote the category of $R$-algebras. 

Let $A\in \CRing$ and $B\in \CRing_A$.
We have a correspondence between the module of derivations $\partial_f\colon A \to B$, which we denote by $\Der(A,B)$, and functions $f\colon A \to B[\eps]/(\eps^2)$ given by 
$$\xymatrix{
% A \ar[r]^f\ar[dr]_{\mbox{alg map}} & B[\eps]/(\eps^2)  \ar[d]^{\pr_1} \\
A \ar[r]^f\ar[dr]_{\partial_f} & B[\eps]/(\eps^2)  \ar[d]^{\pr_1} \\
& B  	
}$$
where $\pr_i\colon B[\eps]/(\eps^2) \to B$ are given by $\pr_0(a+\eps b) = a$ and $\pr_1(a+\eps b) = b$. 
The map from the collection of such $f$'s to the collection of derivations is given by 
 $$ f \mapsto \partial_f = \pr_1 \circ f. $$
If $X$ is a scheme over a ring $R$, we will let $\Der(\OO_X/R)$ denote the sheaf of $R$ linear derivations on $X$; this sheaf is isomorphic to $T_{X/R}$.
% \david{$T_{X/K}$ really is the dual of the sheaf of differentials, right?} \taylor{Yes. Maps from differentials out are also derivations. }

Now for the arithmetic version.  The idea in what follows is to replace $B \mapsto B[\eps]/(\eps^2)$ with other ring schemes to get ``new derivations''. In the same way that derivations are in correspondence with maps to the ring of dual numbers, $\pi$-derivations are defined via maps to rings of truncated witt vectors of length two.

Let $R$ be a finite extension of $\ZZ_p$ with uniformizer $\pi\in R$. Let $q$ denote the cardinality of the residue field of $R$. For an $R$-algebra $A$ we define $W_{\pi,1}(A)$ to be the set $A\times A$ with addition and multiplication rules given by 
 $$ (a_0,a_1) (b_0,b_1) =  (a_0b_0, a_1b_0^q + b_1 a_0^q + \pi a_1b_1), $$
 $$ (a_0,a_1)+(b_0,b_1) = \left(a_0+b_0, a_1+b_1 - \frac{1}{\pi} \sum_{j=1}^{q-1} { q \choose j} a_0^{q-j}b_0^j \right);$$
these are the so-called \defi{ramified witt vectors of length two}.  When the $\pi$ is understood we will just denote this ring by $W_1$. 
% \david{Remind me in person how this differs from the usual Witt vectors of length two}
% \taylor{It does not. But there is a version where you can tensor-up: $R \otimes_{\Z_p} W_{p,1}$, this is not the same.  }

Let $A \in \CRing_R$ and $B\in \CRing_A$, with structure map $g \colon A \to B$. We define a $\pi$-\defi{derivation} to be a function $\delta\colon A \to B$ such that the map 
\[
f\colon A \to W_1(B), \, x \mapsto f(x) := (g(x), \delta(x))
\]
is a ring homomorphism. Given a ring homomorphism $f\colon A \to W_1(B)$, the composition  $\delta_f = \pr_1\circ f$ is a $\pi$-derivation. From the sum and product rules for Witt vectors we may derive the sum, product and identity rules for $p$-derivations.  We denote the collection of $\pi$-\defi{derivations} from $\delta\colon A \to B$ by $\pDer(A,B)$.

\begin{example}
	In the examples below we will let $\pi=p$ a rational prime. 
	\begin{enumerate}
		\item $\delta\colon \Z_p \to \Z_p$ given by $\delta(x) = (x-x^p)/p$;
		\item $\delta\colon  \Z/p^2 \to \Z/p$ given by the same formula. 
		Note now that division by $p$ is a map $p \Z/p^2 \to \Z/p$. 
	\end{enumerate}
\end{example}

Finally note that given a $\pi$-derivation $\delta$, the map $\phi(x) = x^{q} + \pi\delta(x)$ is a lift of the Frobenius (a ring homomorphism whose reduction modulo $\pi$ coincides with a $q$th power map). 

\subsection{Notation for reductions mod powers of primes}
\label{sec:notat-reduct-mod}
% Let $R$ be a finite extension of $W_{p,\infty}(k)$ where $k$ is a subfield of $\FF_p^{\alg}$. 
% Let $\pi$ be a uniformizer of $R$. \david{And $q$ is? }
We start with a field $K$ of characteristic zero, complete under a discrete valuation $v$, with residue field $k$ of characteristic $p > 0$.  We assume $v$ is normalized such that $v(K^{\times}) = \Z$ and we denote by $e := v(p)$ the absolute ramification index.  Let $R$ be the valuation ring of $K$.  Assume now that we are given a prime element $\pi \in R$ which is algebraic over $\Q_p$. Having fixed $K$ and $\pi$ as above we shall define a map $\delta\colon R \to R$ which will play the role of a ``derivation with respect to $\pi$''.  Let $q$ be the cardinality of the residue field of $\Q_p(\pi)$. Then, by standard local field theory, there exists a ring automorphism $\phi\colon R \to R$ that lifts the Frobenius automorphism $F\colon k \to k$, $F(x) := x^q$.
We define the map $\delta \colon R \to R$ by the formula 
$$\delta(x) = \frac{\phi(x) - x^q}{\pi}$$ 
for $x \in R$ We shall usually write $x',x'',\ldots, x^{(n)}$ in place of $\delta(x), \delta^2(x),\ldots, \delta^n(x)$.

There exists a unique lift of the Frobenius $\phi = \phi_{R,\pi}$ which acts as $\phi(\zeta_n) = \zeta_n^q$ (for $(n,q) = 1$) and satisfies $\phi(\pi) = \pi$. We will let 
 $$ R_n = R/\pi^{n+1} $$
and for $X/R$ a scheme we let 
 $$ X_n = X \otimes R_n = X \text{ mod } \pi^{n+1}. $$

\subsection{Absolute and relative Frobenius}
For $X_0/S$ a scheme over a base $S$ of characteristic $p$ we will let $F_{X_0} = F_{X_0,q}$ denote the absolute Frobenius and $F_{X_0/S} = F_{X_0/S,q}$ denote the relative Frobenius. 
They fit into a diagram
 $$
 \xymatrix{ X_0 \ar[dr]^{F_{X_0/k}} \ar@/^2pc/[drr]^{F_{X_0}} & & \\
 &X_0^{(q)} \ar[r] \ar[d] & X_0 \ar[d] \\
 & S \ar[r]^{F_S} & S.
 }
$$

Here $X_0^{(q)} = X_0 \times_{S, F_S} S $ is the Frobenius twist of $X_0$, which is just the pullback of $X_0$ by the Frobenius on the base.
In terms of equations, we simply raise to $q$th power the coefficients of the defining equations of $X_0$.  On sections we have $F_{X_0}^{\#}(f) = f^q$ and $F_S^{\#}(a) = a^q$. When no confusion arises, we may just denote a Frobenius as $F$. 

Let $X$ and $X'$ be schemes or $\pi$-formal schemes over $R$ which lift $X_0$. 
A \defi{lift of the Frobenius} is a morphism 
$$\phi\colon  X \to X' $$
such that $\phi\otimes_R R/\pi \cong F_{X_0}$. 

\subsection{Frobenius derivations}
For $X_0$ a scheme over a field $k$ of characteristic $p$ we define the sheaf $\FDer(\OO_{X_0}) $ of \defi{Frobenius semi-linear derivations} or \defi{$F$-derivations} to be $\FDer(\OO_{X_0})  := F_{X_0}^*T_{X_0/k}$; note that these can be either the $p$-Frobenius or a $p^a$-Frobenius depending on the context. 
It follows directly from the definition that local section $D$ has the property that $D$ acts as
 $$D(xy) = x^qD(y) + D(x) y^q,$$
where $x$ and $y$ are local sections of $\OO_{X_0}$. 

% \david{Help me check that this follows from the definitions}
% \taylor{ 
% Here is a starting point: in the paper of Deligne-Illusie they use that 
%  $$\Hom_{\OO_{X'}}(\Omega^1_{X'/k},F_*\OO_X)$$
% is the sheaf of $D: \OO_{X'} \to F_*\OO_X$ where we have
%  $$ D(ab) = D(a)F_r^{\#}(b) + F_r^{\#}(a) D(b) $$
% where $F=F_r$.
% One needs to use that $F_{abs} = G \circ F_{rel}$. 
% I believe to get the full derivations you just have to compose with $G^{\#}$ where $G$ is the map that I removed the label from from $X' \to X$. 
	
% Let $f:X \to Y$ be a morphism over $k$. 
% Then $\Hom_{\OO_{Y}}(\Omega^1_{Y/k}, f_*\OO_X)$ are the $f^{\#}$-linear derivations. 

% Let me do this during our department meeting and write this out when I get back. 
% Or if you get to this first we can fix it. 
% }

\subsection{Deligne--Illusie classes}\label{S:deligne-illusie}
Let $X/R$ be a smooth scheme. 
As in the above setup, let $\delta\colon R \to R$ be the unique $\pi$-derivation such that the induced Frobenius fixes a chosen uniformizer $\pi$. 
We define the \defi{Deligne--Illusie class} to be the \u{C}ech cohomology class
 $$ \DI_{X_1/R_1}(\delta) = [\delta_i - \delta_j \mod \pi] \in H^1(X_0, F_{X_0}^*T_{X_0/k}) $$
where $\delta_i\colon  \OO_{U_{i,1}} \to \OO_{U_{i,0}}$ are local prolongations of $p$-derivations on the base and $(U_{i} \to X)_{i\in I}$ is a cover by Zariski affine opens with lifts of the $\pi$-derivations.
Such lifts exist locally due the the infinitesimal lifting property.
See, for example, \cite[Lemma 1.3]{Buium1995}. 
When the derivation on the base $R$ is understood we will use the notation
 $$ \DI_{X_1/R_1}(\delta) = \DI(X_1). $$
When we want to signify that $\DI(X_1)$ is an obstruction to lifting the $m$th power Frobenius we use the notation $\DI^m(X_1)$.

\subsection{Kodaira--Spencer classes and three properties of Kodaira--Spencer classes}
Let $X/K$ be a smooth projective variety. 
Let $\partial_K\in \Der(K)$ be a derivation on the base. 
Let $(U_i\to X)_{i\in I}$ be a cover by Zariski opens. 
The \defi{Kodaira--Spencer class} is defined by 
 $$ \KS_{X/K}(\partial_K) = [\partial_i-\partial_j] \in H^1(X,T_{X/K}) $$
where $\partial_i \in \Der(\OO_X(U_i))$ are prolongations of the derivation on the base: $\partial_i\vert_K = \partial_K$. We present three properties which will have arithmetic analogs. 
\subsubsection{Property 1: Representability of sheaf of prolongations of derivations}
The first jet space is defined to be the representative of the sheaf of prolonged derivatives $\Der(\OO_X/(K,\partial_K))$ on $X$:
 $$ \Der(\OO_X/(K,\partial_K)) \cong \Gamma_X(-,J^1(X/(K,\partial_K))).$$
Here $g\colon  J^1(X/(K,\partial_K)) \to X$ is the first jet space on $X$ and the right hand side denotes the sheaf of sections of $g$.\footnote{In the special case that $\partial_K=0$ we have   $J^1(X/(K,\partial_K)) = T_{X/K}$ and the functor of points of $J^1$ is just $X$ composed with the dual number functor; i.e., 
\[
J^1(X/(K,\partial_K))(A) \to X(A[\epsilon]/\epsilon^2).
\]} Local sections of this space are local lifts of the derivation. One may observe that $J^1(X/(K,\partial_K))$ is a torsor under $T_{X/K}$, and is thus classified by $\KS_{X/K}(\partial_K) \in H^1(X,T_{X/K})$
(the difference of two derivations prolonging a derivation on the base field is zero on the base field since they agree there).

\subsubsection{Property 2: Buium--Ehresmann Theorem}
Let $K$ be an algebraically closed field equipped with a derivation $\partial$.
In what follows we let $K^{\partial} = \lbrace a \in K\colon  \partial(a)=0 \rbrace$ denote the field of constants.
The following are equivalent for $X/K$ projective:
\begin{enumerate}
	\item $\KS_{X/K}(\partial)=0$,
	\item $X/K$ admits a global lift of $\partial$, and 
	\item $X \cong X_0 \otimes_{K^{\partial}} K$ for some scheme $X_0$ defined over $K^{\partial}$;
\end{enumerate}
see \cite[Ch II, Section 1]{Buium1986}.

\subsubsection{Property 3: Kodaira--Spencer Compatibility}
In \cite{Faltings1999} it was asked if there exists an arithmetic Kodaira--Spencer class. 
He isolated the following key property: let $K$ be a field with a derivation. 
If $f\colon X \to Y$ is a morphism over $K$ (say smooth or a closed immersion) then 
 $$ f^*\KS_{Y/K}(\partial) = df_* \KS_{X/K}(\partial) \in H^1(X,f^*T_{Y/K}),$$
where 
\[
df\colon T_{X/K} \to f^*T_{Y/K}
\]
is the natural map and $df_*$ is the induced map on cohomology.

\subsection{Three analogous properties for Deligne--Illusie classes }
We now present three properties (one of which is new and stated as a theorem) which are analogs of the three properties for Kodaira--Spencer classes.

\subsubsection{Property 1: Representability of sheaf of prolongations of $p$-derivations}
We now work over $R$ a finite extension of $\ZZ_p$ with prime element $\pi \in R$.
Let $X$ be a $\pi$-formal scheme over $R$ as in section \ref{sec:notat-reduct-mod}.
We define the first $\pi$-jet space (\cite{Buium2005, Buium1995}) to represent the sheaf of $\pi$-derivations on $X$.
More precisely the map $g\colon  J^1(X) \to X$ represents the sheaf of $\pi$-derivations (in characteristic zero).
That is, local sections of $g$ correspond to local lifts of $\pi$-derivations. 
When talking about the first $\pi$-jet space of a scheme we will always mean the first $\pi$-jet space of its $\pi$-formal completion. 

We can consider the above situation modulo $\pi^2$. Here, the sheaf $\pDer(\OO_{X_1},\OO_{X_0})$ of prolongations of the $\pi$-derivation $\delta_{1}\colon  R_1 \to R_0$ is representated by sections of a map
 $$g_0\colon  J^1(X)_0 \to X_0.$$ 
Here $J^1(X)_0$ is the reduction mod $\pi$ of the first arithmetic jet space.
Local sections of $g_0$ correspond to local lifts of the Frobenius on $\OO_{X_1} \to \OO_{X_1}$, or equivalently $\pi$-derivations $\OO_{X_1} \to \OO_{X_0}$.
The scheme $J^1(X)_0$ is a torsor under $F^*T_{X_0}$ whose class is classified by $\DI_{X_1/R_1}(\delta_1)$ (this can be seen by just subtracting two $\pi$-derivations pointwise and obtaining a derivation of the Frobenius). 
 
Locally, the constructions looks as follows: for a ring $A = R\langle X \rangle/(G) = \varprojlim R[X]/(G,\pi^n)$, where $X = (x_1,\ldots, x_m)$ and $G=(f_1,\ldots, f_e)$, we have 
 $$ \OO(J^1(\Spec(A))) = R\langle X,\dot{X}\rangle/(G,\delta(G))$$
where $\delta(G)$ denotes the tuple of formal $\pi$-derivations of the elements $f_1,\ldots,f_e$ which we understand as expanding using the sum and product rules to arrive at elements of $R[X,\dot{X}]$.
For example 
 $$\delta(x^2+rx_1) = 2x_2^q\dot{x}_2 + \delta(r)x_1^q + \dot{x}_1 r^q + \pi\dot{x}_1\delta(r) C_{\pi}(x_2^2,rx_1)$$ 
where $C_{\pi}(a,b) = \frac{a^q+b^q-(a+b)^{q}}{\pi} \in R[a,b]$ is the polynomial in the addition rule for Witt vectors. 
Here the universal formal $\pi$-derivation $\delta\colon  R\langle X \rangle \to R\langle X, \dot{X} \rangle$ prolongs the fixed $\pi$-derivation on the base.  This construction globalizes to give a $\pi$-formal scheme $g\colon J^1(X) \to X$.  

\subsubsection{Property 2: Buium--Ehresmann Theorem/Descent Philosophy}
Recall that 
$$\DI_{X_1/R_1}(\delta_1)=0$$ if and only if $X_1$ has a lift of the Frobenius modulo $\pi^2$. 
In view of the analogy with Buium--Ehresmann theorem this should be viewed as a sort-of descent. 
In fact, Borger defines a category of $\Lambda_p$-schemes where the objects are pairs $(X,\phi_X)$ consisting of schemes or $p$-formal schemes together with lifts of the Frobenius and whose morphisms $(X,\phi_X) \to (Y,\phi_Y)$ are morphisms $f\colon X\to Y$ which are equivariant with respect to $\phi_X$ and $\phi_Y$. 
We think of this as a sort of descent to the field with one element in view of \cite{Borger2009}.

%This analogy is useful particularly in the context of Serre--Tate theory where canonical lifts can be viewed as $\Lambda_p$-points of local moduli. 
%This observation appears to be new. 

\subsubsection{Property 3: Deligne--Illusie Compatibility}

In the present paper we prove the following.
\begin{thm}\label{T:compatibility}
Let $f\colon X\to Y$ be a morphism of formally smooth $\pi$-formal schemes over $R$ (a finite extension of $\Z_p$ with specified prime element $\pi$). If $f$ is smooth or a closed immersion then 
\begin{equation}\label{E:compatibility}
df_* \DI_{X_1/R_1}(\delta_1) = f^* \DI_{Y_1/R_1}(\delta_1) \in H^1(X,f^*F^*T_{X_0/R_0}).
\end{equation}	
\end{thm}

This property is new and is proved in \S \ref{S:proofs}. The proof uses affine bundle structures of $J^1(X/R)$---the first $\pi$-arithmetic jet space of Buium, the fact that smooth morphisms locally decompose as \'etale morphism followed by projections from an affine space, and properties of jet spaces and \'etale morphisms $X \to Z$ of $p$-formal schemes $J^1(X) \cong J^1(Z) \times_Z X $ to build ``local Frobenius compatibility data''.

\section{Proof of compatibility}\label{S:proofs}
In what follows we will fix $R$ a finite extension of $\Z_p$ with uniformizer $\pi$ and residue field $k$ of cardinality $q$. 
We will fix a $\pi$-derivation on the base.

\begin{defn}
	\begin{enumerate}
    \item []
		\item A morphism of $R$ schemes $f\colon X \to Y$ is \defi{Deligne--Illusie compatible} provided 
		$$ df_*(\DI(X_1) ) = f^* \DI(Y_1)  \in H^1(X_0,  f^*\FT_{Y_0}  ). $$ 
		\item \label{defn: frobenius compatibility}
		Let $f\colon  X\to Y$ be a morphism of $\pi$-formal schemes.
		By locally \defi{local Frobenius compatibility data} for $f$ we will mean two covers 
		$$(U_i\to X)_{i\in I} \mbox{ and } (V_i\to Y)_{i\in I}$$ 
		with lifts of the Frobenius $\phi_{U_i}$ and $\phi_{V_i}$ (with the second cover possibly having repeat open sets) such that for each $i$,
$$f(U_i)\subset V_i$$ and $f|_{U_i}$ is compatible with $\phi_{U_i}$ and $\phi_{V_i}$.
		\item 
			If $f$ admits local Frobenius compatibility data we will say $f$ is \defi{locally Frobenius compatible}. 
	\end{enumerate}
	
\end{defn}

\begin{lem}\label{lem: Frobenius compatibility}
	Let $f\colon X \to Y$ is be a morphism of smooth $\pi$-formal schemes over $\Spf(R)$. 
	\begin{enumerate}
		\item \label{lem: compatible lifts of Frobenius for closed immersions} If $f$ is a closed immersion then $f$ is locally Frobenius compatible. 
		\item \label{lem: Frobenius compatible part 1} If $f$ is \'etale then $f$ is locally Frobenius compatible. 
		\item \label{lem: Frobenius compatible part 2} If $f$ is a projection of the form $ \AA^n_Y \to Y$ then $f$ is locally Frobenius compatible.
	\end{enumerate}
\end{lem}

In the proofs, we will repeatedly use the fact that a scheme $X$ admits a Frobeinus lift if and only if the map $J^1(X) \to X$ admits a section, and that two lifts are compatible if and only if the induced diagram
		\begin{equation}\label{EQ:lifting-diagram}
		\xymatrix{
      J^1(X)\ar[d] \ar[r] & J^1(Y)\ar[d]\\
			X \ar[r] \ar@/_/[u] & Y \ar@/_/[u]
		}
		\end{equation}
commutes. 

\begin{proof}

We begin with case \ref{lem: compatible lifts of Frobenius for closed immersions}. We will work with $\pi$-formal schemes and omit the hats. 
		Let $X$ have dimension $n$ and $Y$ have dimension $n+m$.
		The problem is affine local, so by \cite[Chapter 3, Proposition 3.13, P.~75]{Buium2005} we may assume without loss of generality that $X$ and $Y$ are affine and that $J^1(X) \cong X \times \AA^n$ and $J^1(Y) \cong Y \times \AA^{n+m}$.		Compatible lifts of the Frobenius $\phi_X$ and $\phi_Y$ are thus equivalent to compatible sections of the diagram
		\begin{equation}
		\xymatrix{
      X \times \AA^{n}\ar[d] \ar[r] & Y \times \AA^{n+m}\ar[d]\\
			X \ar[r]& Y.
		}
		\end{equation}
On coordinate rings,  the map $X \times \AA^{n} \to  Y \times \AA^{n+m}$ is given by a map
\[
			\OO(Y)\langle s_1,\ldots,s_{n+m}\rangle   \xrightarrow{\alpha}			\OO(X) \langle t_1,\ldots, t_n \rangle 
\]
where the $s_i$ and $t_j$ are coordinates on each affine space, and our desired sections correspond to a commutative diagram
		\begin{equation}\label{eqn: closed immersion diagram}
		\xymatrix{
			\OO(Y)\langle s_1,\ldots,s_{n+m}\rangle \ar[d]^{\sigma_Y } \ar[r]^{\alpha} 
&			\OO(X) \langle t_1,\ldots, t_n \rangle \ar[d]^{\sigma_X} \\
 \OO(Y) \ar[r]^{\beta} 
& \OO(X) 
		}
		\end{equation}
		where $\sigma_Y$ and $\sigma_X$ are the natural maps given by $\sigma_Y(s_i) = \delta_Y(s_i)$ and $\sigma_X(t_i) = \delta_X(t_i)$ where $\delta_X$ and $\delta_Y$ are the $\pi$-derivations associated to $\phi_X$ and $\phi_Y$ (c.f. \cite[Chapter 3, section 3.2]{Buium2005}). 
		% As a reminder, angle brackets mean $\pi$-adic completion. 
		Observe that the map $\alpha$ is determined by a formula of the form
		$$ \alpha(s_i) = \sum_{J} a_{i,J}t^J, \ \ \ 1 \leq i \leq m$$
		where $J = (j_1,\ldots,j_n) \in \NN^n$, $t^J = \prod t_i^{j_i}$ is multi-index notation, and $a_{i,J} \in \OO(X)$ $\pi$-adically tend to zero as $\vert J\vert \to \infty$.
		
		Suppose $\sigma_X(t_i)$ is defined by $\sigma_X(t_i) = A_i \in \OO(X)$ for some choices of $A_i \in \OO(X)$.
		We will prove that there exists a lift of the Frobenius of $Y$ which is compatible with this one.
		Observe the compatibility condition $\alpha \circ \phi_Y = \phi_X \circ \alpha$ implies $\beta \circ \sigma_Y = \sigma_X \circ \alpha,$ which implies that 
		\begin{equation*}
		\beta \sigma_Y(s_i) = \sum_{J} a_{i,J} A^J := \overline{B_i}.
		\end{equation*}
		Here $A = (A_1,\ldots, A_m)$. 
		Constructing $\sigma_Y$ to make the diagram \eqref{eqn: closed immersion diagram} commute is now simple: 
		for any $B_i \in \OO(Y)$ with image $\overline{B}_i$ in $\OO(X)$, the morphism $\sigma_Y$ defined by 
		$$ \sigma_Y(s_i) = B_i $$
		works (i.e.~defines a commutative diagram).
		Note that such $B_i$ always exist because $\OO(Y)\to \OO(X)$ was assumed to be surjective.

   Next we prove the second claim.
		Suppose $f$ is \'etale.  
		By \cite[Chapter 3, Corollary 3.16, p.~77]{Buium2005} we have
		\begin{equation}\label{eqn: etale gives isomorphism}
		J^1(X) \cong X \times_Y J^1(Y) 
		\end{equation}
		as $\pi$-formal schemes. In this case, the diagram
		\begin{equation}
		\xymatrix{
      J^1(X) = X \times_Y J^1(Y) \ar[d] \ar[r] & J^1(Y)\ar[d]\\
			X \ar[r]& Y 
		}
		\end{equation}
    is cartesian, and given a section of $\sigma_Y \colon J^1(Y) \to Y$ we can simply take $\sigma_X$ to be $(\id, \sigma_Y)$.

		% In fact, picking a local section $\sigma^Y$ of $Y$, is the same as specifying a finite number of elements $\OO_Y$ (c.f. \cite[Proposition 3.13, pg 75]{Buium2005}).
		% By \eqref{eqn: etale gives isomorphism}, we can find a lift using the exact same coordinates only pulling them back. \footnote{
		% 	The isomorphism $J^1(X)\cong X \times_Y J^1(Y)$ uses the infinitesimal lifting property and representability of $\pi$-derivations as morphism to truncated Witt vectors.}

   For the third claim, 
		% We work with $\pi$-formal schemes and omit the hats. 
		let $m = \dim(Y)$. While it is not in general true that $J^1(X_1 \times X_2) \cong J^1(X_1) \times J^1(X_2)$, this isomorphism does hold if $X_2$ is affine space. We consider the diagram
		\begin{equation*}
		\xymatrix{
      J^1(X) \cong  J^1(Y) \times  J^1(\AA^n) \ar[d] \ar[r] & J^1(Y)\ar[d]\\
			X \ar[r]& Y. 
		}
		\end{equation*}
    Since $J^1(\AA^n) \cong \AA^{2n}$, any section of $Y \to J^1(X)$ extends to a section of $Y \to J^1(X)$, completing the proof.
		% We can assume without loss of generality that both sides admit a lift of the Frobenius.
		% We have 
		% $$\OO(J^1(Y)) = \OO(U)\langle t_1,\ldots, t_m \rangle,$$
		% $$\OO(J^1(Y \times \AA^n)) = \OO(Y)\langle t_1,\ldots, t_m\rangle \langle s_1,\ldots, s_n, r_1,\ldots, r_n \rangle,$$
		% and we can extend any $\pi$-derivation on $\OO(Y)$ to a $\pi$-derivation on the product.  
\end{proof}

\begin{lem}\label{lem: DI compat reductions} 
  The following are true.
	\begin{enumerate}
		\item \label{lem: Frobenius compatible implies Deligne--Illusie compatible} If $f\colon X \to Y$ admits local Frobenius compatibility data, it is Deligne--Illusie compatible. 
		
		\item \label{lem: composition of compatible is compatible} If $f\colon  X \to Z$ is Deligne--Illusie compatible and $g\colon  Z \to Y$ is Deligne--Illusie compatible then their composition is.
	\end{enumerate}
\end{lem}
\begin{proof}
  We will work $\pi$-formally and omit hats everywhere. To begin the proof of the first claim, we fix local Frobenius compatibility data (Definition~\ref{defn: frobenius compatibility}): i.e., we fix open covers $(U_i\to X)_{i\in I}$ and $(V_i \to Y)_{i\in I}$ such that $f(U_i) \subset V_i$ together 
		with $ \phi^X_i\colon  \OO(U_i) \to \OO(U_i)$ and $ \phi^Y_i\colon  \OO(V_i) \to \OO(V_i) $ such that $\phi_X^if^{\#} = f^{\#}\phi_Y^i$.
		Observe that this last condition is equivalent to $\delta^X_if^{\#} = \delta_Y^i f^{\#}$ as elements of $\pDer(\OO_Y, f_*\OO_X)(U_i)$.
		This implies for each $U_{ij} = U_i \cap U_j$ we have 
		\begin{equation}\label{eqn: local DI compat}
		D_{ij}^X f^{\#}=f^{\#} D_{ij}^Y \in \FDer(\OO_Y(V_{ij}), f_*\OO_X(U_{ij})), 
		\end{equation}
		where $D_{ij}^X := \delta^X_i - \delta^X_j$ and $D_{ij}^Y := \delta^Y_i - \delta_Y^J$. Note that the right hand side of \eqref{eqn: local DI compat} induces $df \DI(X)$ and the right hand side of \eqref{eqn: local DI compat} induces $f^*\DI(Y)$.
\footnote{Since $\OO(f^{-1}(U_{ij})) \to \OO(V_{ij}) $ we may view this as giving a map on $X$ and hence giving a cocycle for a sheaf on $X$.}
		\footnote{In general, for $F$ a quasi-cohrent sheaf on $Y$, the map $f^*\colon H^i(Y,F) \to H^i(X,f^*F)$ can be performed locally by just identifying sections of $F$ with sections of $f^*F$ with new coefficients.} 
		
		The proof of the second claim requires the identities
		\begin{eqnarray*}
			d(g\circ h) = h^*(dg_*) \circ dh_*, & (g\circ h)^* = h^* g^*.
		\end{eqnarray*}
		It then follows that
		\begin{eqnarray*}
			f^*\DI(Y) &=& (g \circ h)^* \DI(Y) \\
			&=& h^* g^* \DI(Y) \\ 
			&=& h^* (dg_* \DI(Z) ) \\
			&=& (h^*dg_*)( h^* \DI(Z)) \\
			&=& (h^*dg_*)( dh \DI(X)) \\
			&=& df_* \DI(X).
		\end{eqnarray*}
		The fourth equality follows from the diagram
		$$\xymatrix{
			H^i(Z,\FDer(Z)) \ar[r]^{dg_*} \ar[d]_{h^*} & H^i(Z, g^* \FDer(Y)) \ar[d]^{h*} \\
			H^i(X,h^*\FDer(Z)) \ar[r]^{h^*dg_*} & H^i(X,h^*g^* \FDer(Y) ). 
		}
		$$ 
\end{proof}

\begin{thm}\label{thm: smooth maps are DI compatible}
	Let $f\colon  X_1 \to Y_1$ be a smooth morphism of smooth $R_1$-schemes. Then
	$$ df_*( \DI(X_1) ) = f^*( \DI(Y_1) ) \in H^1(X_0, f^*\FT_{Y_0}). $$
\end{thm}
\begin{proof}

	We first prove the theorem locally and assume we can factor the morphism $f\colon X\to Y$ as
		$$ X \to \AA^n_{Y} \to Y, $$
		where the first map is \'etale and the second map is the standard projection (see e.g.~\cite[Tag 039P]{stacks-project}). 
		This can be done locally where by ``locally'' we mean that there exists a cover by affine open subsets $X' \subset X$ and $Y'\subset Y$ with $f(X') \subset Y'$ with this factorization. 
		
		We will now express $f$ as a composition of Deligne--Illusie compatible morphisms.
		We apply Lemma~\ref{lem: Frobenius compatibility} part~\ref{lem: Frobenius compatible part 1} and Lemma~\ref{lem: Frobenius compatibility} part~\ref{lem: Frobenius compatible part 2} together with Lemma~\ref{lem: DI compat reductions} part~\ref{lem: Frobenius compatible implies Deligne--Illusie compatible} to get the outer morphisms of the composition to be Deligne--Illusie compatible.
		Lemma~\ref{lem: DI compat reductions} part \ref{lem: composition of compatible is compatible} says the composition of compatible morphisms is compatible.
		
We now show compatibility globally. Consider a covering $(U_{i,0} \to X_0)_{i\in I}$ such that
		$$ df_*(\DI(X_1)) \vert_{U_{i,0}} = f^*(\DI(Y_1)) \vert_{U_{i,0}} \in \underline{H}^1(X_0, f^*\FT_{Y_0})(U_{i,0}).$$
		Putting these together gives an element 
		$$ c \in H^0(X_0, \underline{H}^1(f^*\FT_{Y_0})).$$
		The comparison between the cohomology sheaf $\underline{H}^1(X_0,f^*\FT_{Y_0})$ and the cohomology $H^1(X_0,f^*\FT_{Y_0})$ comes from the low degree exact sequence of the spectral sequence comparing sheafy cohomology and cohomology (see for example \cite[01ES]{stacks-project} for the spectral sequence). The convergent spectral sequence is given by $$E_2^{i,j} = H^i(X_0, \underline{H}^j(f^*\FT_{Y_0})) \implies H^{i+j}(X_0,f^*\FT_{Y_0}) $$
		and the low degree exact sequence gives 
		\begin{eqnarray*}
			0 &\to& H^1(X_0, \underline{H}^0(\FT_{Y_0})) \to H^1(X_0,f^*\FT_{Y_0}) \to H^0(X, \underline{H}^1(f^*\FT_{Y_0})) \\
			&\to& H^1(X_0, \underline{H}^0(f^*\FT_{Y_0})) \to H^2(X_0, f^*\FT_{Y_0}) 
		\end{eqnarray*}
		which reduces to 
		$$ 0 \to H^1(X_0,f^*\FT_{Y_0}) \to H^0(X,\underline{H}^1(f^*\FT_{Y_0})) \to H^2(X_0,f^*\FT_{Y_0}) \to 0. $$
		By local compatibility we have that $f^*\DI(Y_1)$ and $df \DI(X_1)$ in $H^1(X_0,f^*\FT_{Y_0})$ map to the same element in $H^0(X_0,\underline{H}^1(f^*\FT_{Y_0}))$; since the map 
\[
H^1(X_0, f^*\FT_{Y_0}) \to H^0(X_0,\underline{H}^1(f^*\FT_{Y_0}))
\]
 is injective, the desired equality follows. 
\end{proof}

\section{Applications}\label{S:applications}

	\subsection{The Wittfinitesimal Torelli problem}
Let $R$ be the valuation ring of a subfield of $\C_p$. We wish to study \eqref{E:compatibility} in the special case that $X=C$ is a (pointed) curve of genus $g\geq 2$ over $R$ and $Y = \Jac_C = A$ is its Jacobian. The compatibility condition for the Abel--Jacobi map $j\colon C\to A$ in \eqref{E:compatibility} can be intepreted as saying the ``wittfinitesimal torelli map''
\begin{equation}\label{E:wt}
dj_*\colon  H^1(C_0, F^*T_{C_0}) \to H^1(A_0, F^*T_{A_0})
\end{equation}
carries $\DI_{C_1/R_1}(\delta_1)$ to $\DI_{A_1/R_1}(\delta_1)$. 
In the Kodaira--Spencer setting, the map \eqref{E:wt} is injective outside the hyperelliptic locus \cite{Oort1979}.
This map is also has the geometric interpretation as the tangent to the Torelli map --- the Torelli map being the map between the moduli space of curves of genus $g$ and the moduli space of principally polarized abelian varieties. 
The prospect of such injectivity is interesting in our setting as it is a theorem of Raynaud that for $g\geq 2$ the $\DI_{C_1/R_1}(\delta_1) \neq 0$ (see \cite{Dupuy2014} for a generalization of Raynaud's result). 
If \eqref{E:wt} were injective this would imply that $A_1$ would not have a lift of the Frobenius. 
% (Note, or course, that some abelian varieties do admit lifts of Frobenius; for instance, on a supersingular elliptic curve, or a Fermat curve)
Unfortunately (or fortunately), for dimension reasons this map is not injective.
\begin{lem}\label{L:non-injective}
	The map \eqref{E:wt} is not injective if $(2p+1)(g-1)>g^2$. 
\end{lem}
\begin{proof}
	This is follows from dimension counting and the rank-nullity theorem of elementary linear algebra. 
	By Riemann--Roch, $h^1(C_0,F^*T_{C_0}) = (2p+1)(g-1)$. Since $T_{A_0} \cong \OO_{A_0}^g$ we have $F^*T_{A_0} \cong F^*( \OO_{A_0}^g ) \cong \OO_{A_0}^g$ and hence $h^1(A_0,F^*T_{A_0}) = g h^1(A_0,\OO_{A_0}) = g^2$.
	This shows the map is not injective when $(2p+1)(g-1)>g^2$. 
\end{proof}

Although the map \eqref{E:wt} is not injective is it still interesting to determine when 
$$ dj_*( \DI_{C_1/R_1}(\delta_1) ) \neq 0,$$
as this gives a criterion to check that a given Jacobian doesn't admit a lift of the Frobenius. 
%In fact, the situation isn't actually as bad as Lemma~\ref{L:non-injective} would indicate. We actually have $\DI^m_{C_1/R_1)(\delta_1) \in H^1_{dR}(F^*T_{C_0}, \nabla^{can})$ which has a fixed dimension with $p$.
This is related to conjectures of Coleman about Jacobians with complex multiplication. 

\begin{rem}
	There are positive results in this direction which say that on an open subset of the ordinary locus in the moduli space of curves of genus $g\geq 2$, the canonical lift of a Jacobian is no longer a Jacobian by showing canonical lifts of Jacobians don't have lifts of the Frobenius moduli $p^2$.
	This was proved independently in the two papers \cite{Dwork1986, Oort1986}.
\end{rem}

%%%%%%%%%%%%%%%%%%%%%%%%%%%%%%%%%%%%%%%%%%%%%%%%%%%%%
\subsection{A conjecture of Coleman}
%%%%%%%%%%%%%%%%%%%%%%%%%%%%%%%%%%%%%%%%%%%%%%%%%%%%%
For the definition of a CM field we refer the reader to \cite[1.3.3]{Chai2014}.
By a CM algebra, we will mean a product of CM fields. 

In what follows, for an abelian variety $A$ over a ring $R$ we let $\End(A/R)$ denote the ring of endomorphisms of $A$ as an $R$-scheme and we will let $\End^0(A/R) = \End(A/R)_{\Q} = \End(A/R) \otimes \Q$. 

\begin{defn}
\label{def:CM}
	Let $A$ be an abelian scheme over a ring $R$, and let $g$ the relative dimension of $A \to \Spec R$.
\begin{enumerate}
	\item Let $F$ be a CM field. 
	By a \defi{complex multiplication by $F$} on $A$ we will mean an injective map $j\colon  F \to \End^0(A/R)$. 
	\item If there exists a semisimple $\Q$-subalgebra $P \subset \End^0(A/R)$ with $\dim_{\Q}(P) = 2g$ then we say $A$ has \defi{sufficiently many complex multiplications} abbreviated $\smCM$. (We will be mostly interested in the case when $j\colon  F \to \End^0(A/R)$, with $F$ a field of degree $2g$ over $\Q$ for this paper.) 
\end{enumerate}
\end{defn}

\begin{rem}
	The following facts can be found in (say) \cite{Moonen2017}.
	\begin{enumerate}
		\item Let $A/K$ be an abelian variety over a field of characteristic zero. If $j\colon F \to \End^0(A/K)$ is an embedding with $\dim_{\Q}(F) = 2g$ then $\End^0(A/K)$ is commutative. 
		\item In both characteristic $p$ and characteristic zero there exist abelian varieties with noncommutative endomorphism algebras. 
		\item There do not exist ordinary simple abelian varieties over finite fields with noncommutative $\End^0(A)$.
		In fact if the $p$-rank $f$ has $f\geq g-1$ then $\End(A/\FF_q)$ is commutative.
		\item For $A_0/\FF_q$ an abelian variety with $q=p^a$ the center of the endomorphism algebra is generated by the Frobenius: $Z(\End(A/\FF_q)) = \ZZ[F_{A_0,q}]$. 
		In every case but the case that $A_0$ is a special type of supersingular elliptic curve, this center is an imaginary quadratic field. 
		\item Terminology for $\smCM$ can vary. For example \cite[around Proposition 3.5]{DeJong1991} calls abelian schemes with $\smCM$ abelian schemes of CM type. 
	\end{enumerate}
\end{rem}

\begin{conj}[Coleman, \cite{Coleman:torsion-points-galois-book}]
	For $C/\C$ of genus $g\geq 8$ there are only finitely many $C$ such that $\Jac_C$ has sufficiently many complex multiplications. 
\end{conj}

\begin{rem}
	The conjecture was originally given by Coleman for $g\geq 4$ which was proven false in \cite{DeJong1991}. 
	The version stated here can be found in \cite{Chai2012}.	
\end{rem}

To explain how the Coleman conjecture is related to Deligne--Illusie classes we first need to recall some facts about Serre--Tate theory, canonical and quasi-canonical lifts, and some basic CM theory.

%%%%%%%%%%%%%%%%%%%%%%%%%%%%%%%%%%%%%%%
\subsection{Serre--Tate theory}
%%%%%%%%%%%%%%%%%%%%%%%%%%%%%%%%%%%%%%%
The following is found in \cite{DeJong1991} and is based on work found in \cite{Messing1972,Katz1981}.
Let $k$ be contained in $\overline{\FF}_p$. 
Let $R$ a complete local ring with residue field $k$. 
Recall that the Serre--Tate theorem states that formal deformations of Abelian schemes are in correspondence with pairings on associated Tate modules; i.e., there is a bijection
\begin{align*}
\DefAS_{A_0/R_0}(R) &\xrightarrow{\sim} \Hom_{\Z_p}(T_pA_0 \otimes T_pA_0^t, \widehat{\G}_m(R)) \\
A &\mapsto q_A.
\end{align*}
Given a lift $A$, we call $q_A$ the associated \defi{Serre--Tate pairing}. 

Suppose that $A_0$ is ordinary of dimension $g$ (so that $A_0[p](\Fbar_p) \cong \mathbf{F}_p^g$). We may fix a basis $\lbrace v_i \rbrace_{i=1}^g$ of $T_pA_0(k)$ and $\lbrace w_j \rbrace_{j=0}^g$ of $T_pA_0^t(k)$.
The \defi{Serre--Tate parameters} (relative to the chosen bases) are
\begin{equation}
q_{i,j}(A) = q_A(v_i,w_j) \in \widehat{\G}_m(R).
\end{equation}

\begin{defn}
	Let $A_0/k$ be an ordinary abelian scheme.
	Let  $R$ be a complete local ring with residue field $k$ and maximal ideal $\mm$. 
We say that an ordinary $A \in \Def_{A_0/k}(R)$ is a \defi{canonical lift} of $A_0$ if for all $\alpha \in T_pA_0(\overline{k})$ and $\beta\in T_pA^t_0(\overline{k})$ we have 
		$$ q_A(\alpha,\beta) = 1 \in \widehat{\G}_m(R) = 1+\mm.$$
(Such a lift is unique and will be denoted by $\widetilde{A_0}$.) 

We say that $A \in \Def_{A_0/k}(R)$ is \defi{quasi-canonical} if  there exist some natural number $m$ such that for all $\alpha\in T_pA_0(k),\beta \in T_pA_0^t(k)$ we have 
		$$ q_A(\alpha,\beta)^m = 1 \in \widehat{\G}_m(R) = 1+\mm.$$
		We denote the set of quasi-canonical lifts by $\qCL(R)$. 

\end{defn}

\begin{lem}[See e.g.~{\cite[Section 3]{DeJong1991}}]
Suppose that $A$ is a lift of $A_0$ to $R$. Then the following are true.
\begin{enumerate}
	\item If $A \in \qCL(R)$, then the power of $m$ in the definition of quasi-canonical can be taken to be a power of $p$;
	\item being quasi-canonical is an isogeny invariant;
	\item $\End(A_0) = \End(\widetilde{A_0})$;
	\item if $A \in \qCL(R)$, then $\End(A_0)_{\Q} \cong \End(A)_{\Q}$;
	\item $A = \widetilde{A_0}$ if and only if $A$ has a lift of the Frobenius;
	\item $A$ is a quasi-canonical lift if and only if $A$ has a lift of a power of the Frobenius. 
\end{enumerate}	
\end{lem}
To derive these facts we basically apply the following lifting lemma repeatedly. 
For details we refer the reader to \cite[Section 3]{DeJong1991} for a well-written treatment.
\begin{lem}[See e.g.~{\cite[Section 3]{DeJong1991}}]
	Let $A_0$ and $B_0$ be abelian varieties over $k$ with formal lifts $A,B$ to $R$.
	Let $f_0\colon A_0 \to B_0$ be a morphism of abelian schemes. Then
	$$ \mbox{$f_0$ lifts to a morphism $A \to B$ } \iff q_A(f_0(\alpha),\beta) = q_B(\alpha,f_0^t(\beta)),$$
	for all $\alpha \in T_pB_0$ and all $\beta\in T_pA_0$. 
\end{lem}

%%%%%%%%%%%%%%%%%%%%%%%%%%%%%%%%%%%%%%%%%%%
\subsection{Applications to Coleman's Conjecture on CM Jacobians}\label{S:applications-to-coleman}
%%%%%%%%%%%%%%%%%%%%%%%%%%%%%%%%%%%%%%%%%%%

While it is well-known that the canonical lift of an abelian variety has smCM, the converse is less well-known (c.f.~\cite[\S 4.1]{Buium2009}).
\begin{thm}\label{T:complex-multiplication}
	Let $R$ be a finite extension of $W(k_0)$ where $k_0\subset \overline{\FF}_p$.
	Let $k$ be the residue field of $R$. 
	Let $A/R$ be an abelian scheme of relative dimension $g$.  
	
	\begin{enumerate}
		\item \label{I:qCL-implies-CM} If $A$ is $\qCL$, then $A$ has smCM.
% a complex multiplication $j\colon  \Q(F_{A_k}) \to \End^0(A)$ where $\Q(F_{A_k})$ is a CM field of dimension $2g$.
		\item \label{I:CM-implies-LOF} If $A/R$ has smCM, 
% a complex multiplication $ j\colon  F \to \End^0(A/R)$ such that $\dim_{\Q}(F) =2g$,
      then $A$ has a lift of the Frobenius. 
	\end{enumerate}
\end{thm}

\begin{proof}
(Compare to \cite[Proposition 3.5]{DeJong1991}.) For (\ref{I:qCL-implies-CM}), the non-simple case follows from the simple case. Suppose that $A_0/k$ is ordinary and simple and $A/R$ is a quasi-canonical lift. 
		Then $\End^0(A_0/k) = \Q(\pi)$, where $\pi = F_{A_0} \in \End(A_0/k)$\footnote{One just needs that the $p$-rank $f$ satisfies $f\geq g-1$ for this part; see
		\cite[5.9]{Oort2008}.}
 is the absolute Frobenius, and has suffiently many complex multiplications.
		The canonical lift $\widetilde{A_0}$ has $\End(A_0/k) = \End(\widetilde{A_0}/R)$ and since the endomorphism algebra is an isogeny invariant and all quasi-canonical lifts are isogenous we have $\End^0(A/\overline{R'})=\Q(\pi)$ as well.

    The proof of the second part is similar to (c.f. \cite[pg 511]{Serre1968}).
		Let $j\colon F \to \End^0(A/R)$ be the complex multiplication.
		Let $k$ be the residue field of $R$. 
		The specialization map $\End(A/R) \to \End(A_0/k)$ is injective. 
		The Frobenius $F_{A_k}$ commutes with every endomorphism so is in the center (and generates the center in the simple case). 
		Since the reduction of $j(F) \cap \End(A/R)$ is its own centralizer \cite[Corollary 1 of Theorem 5 + ]{Serre1968}\footnote{Some work needs to be done here to check the proof carries through in the non-simple case.}, its reduction contains the center of $\End(A_k/k)$ which proves that there exists some $\pi_A \in j(F) \cap \End(A/R)$ mapping to $F_{A_k} \in \End(A_k/k)$. 
% 		\iffalse
% Let $A$ be an abelian scheme over $R$. 
% Let $E = \End^0(A/R)$.
% 		Let $E_0 = \End^0(A_k/k)$.
% 		Identify $F$ with $j(F)$. 
% 		If $[F:\Q]=2g$ then $F = Z(E)$ \cite[1.3.1.1]{Chai2014}.
% 		By the specialization injection, we also have $F \subset E_0$ which shows $F \cong Z(E_0)=\Q(F_{A_k})$. 
% 		\fi
\end{proof}

	The following remark explains why considering lifts of the Frobenius and smCM abelian schemes allow us to study lifts of non-ordinary abelian varieties. 
\begin{rem}
	Note that in the above theorem having $j\colon  F \to \End^0(A/R)$ and being a quasi-canonical lift are not equivalent---the hypotheses of being $\qCL$ (in \ref{T:complex-multiplication}(\ref{I:qCL-implies-CM})), while including a lift of a power of the Frobenius, has an ordinarity assumption baked into it.
	On the other hand, the CM hypothesis in \ref{T:complex-multiplication}(\ref{I:CM-implies-LOF}) may include an abelian scheme $A/R$ whose reduction is not ordinary.
	This shows that a variety with a lift of a $q$-Frobenius is not necessarily a quasi-canonical lift.

%	If $A/K$ is an abelian variety with good reduction at a prime $p$ the Newton polygon (and hence $p$-rank) of the reduction is determined by so Shimura-Taniyama formula \cite[\S 2.1.5.1]{Chai2014}.
\end{rem}

\begin{rem}
	The following remark explains what to do when working over fields. 
	Suppose $K$ is a characteristic zero field and $A/K$ is an abelian variety of dimension $g$ with CM by a field $F$ of dimension $2g$ over $\Q$. 
	Without loss of generality we can take $K$ to be a number field. 
	
	Let $v$ be a place of $K$ for which $A$ has good reduction.
	Let $K_v$ be the completion of $K$ with respect to $v$ and $R_v$ the ring of integers of $K_v$. 
	One may consider the N\'eron model $\mathcal{A}$ of $A_{K_v}$ over $R_v$. 
	Such an $A_{R_v}$ satisfies the hypotheses of \ref{T:complex-multiplication}(\ref{I:CM-implies-LOF}).
	
	 Also, the Main Theorem of Complex Multiplication \cite{Serre1968} (see also \cite[Appendix A.2]{Conrad2007}) states that there exists a finite extension $K'/K$ (which can be made explicit) such that $A_{K'}/K'$ has good reduction at every place. 
\end{rem}

Given the relationship between lifts of the Frobenius and CM abelian varieties it is now very interesting to determine the smallest $r$ such that every abelian variety $A/k_0$ of dimension $g$ with CM $j\colon \OO \to \End(A/\C)$ has a lift of the $p^r$-Frobenius in some integral model.

\begin{example}
	The case $g=1$ is due to Deuring, see \cite[XIII, page 294, proof of (iii)]{Cassels1967}; in this case one has a $q^r$-power Frobenius for $r=2$.
\end{example} 

In what follows we assume the reader is familiar with terminology from the theory of complex multiplication. 
We point to \cite{Lang2012} as a readable general reference.

Fix a rational prime $p$ and an abelian variety $A/L$ of dimension $g$ where $L$ is a number-field. 
Suppose that $A/\C$ has a complex multiplication $j\colon F\to \End^0(A/\C)$ where $[F:\Q]=2g$. 
Let $\Phi$ be the CM type of $F$ obtained by looking at the tangent space of $A$.
We know that base changing to $L' = F^{\star} L$, where $F^{\star}$ is the reflex field of $(F,\Phi)$, 
that $\End^0(A/\C) = \End^0(A/L') \cong F$. 
To get good reduction at every place, by \cite{Serre1968}, we may take a further extension $K/L'$.
% \iffalse \footnote{ 
% 	Let $\rho_l\colon  I_{\bar{\pp}/\pp} \to \End(T_l)$ be the $l$-adic representation of the inertia group. 
% 	These groups have finite commutative image for CM abelian varieties and the extension $K$ is determined 
% } \fi
Let $\mathcal{A}/\OO_K$ be the N\'{e}ron model of of $A/K$. 
Theorem~\ref{T:complex-multiplication}(\ref{I:CM-implies-LOF}) tells us now that for every prime $\pfrak$ of $\OO_K$, there exists some $\pi_{\mathcal{A}} \in\End(\mathcal{A}/\OO_K)$ lifting the $q$-Frobenius in $F_{\mathcal{A}_{\pfrak},q} = F_{\mathcal{A}_{\pfrak}/\kappa(\pfrak),q}\in \End(\mathcal{A}_{\pfrak}/\kappa(\pfrak))$. 
Here $q=p^r=\#\kappa(\pfrak) = p^{[\kappa(\pfrak):\FF_p]}$.
As a bound for $r = [\kappa(\pfrak):\FF_p]$, we clearly have 
 $$r\leq [K:\Q],$$ 
and so this bound is governed by the extension 
 $$\Q \subset L \subset L'=F^{\star}L \subset K.$$
The extension $L \subset \Qbar$ is the field of definition of the CM abelian variety, the bound $L\subset L'$ pertains to the CM field, and the extension $L' \subset K$ has to do with inertia of $L'/\Q$ at $\pp \cap \OO_{k}$. 

By \cite{Serre1968} (see also \cite[5.2]{Katz2016}) an abelian variety $A/L'$ with semistable reduction at $p\neq 2$ at $K = k(A[l])$ where $l \neq p$. 
For a principally polarized abelian variety this field has Galois group contained in $\GSp_{2g}(\FF_{\ell})$; when $A$ has CM, it has good reduction over $K$, and the Galois group is an abelian subgroup of $\GSp_{2g}(\FF_{\ell})$. We note that 
$$ \#\GSp_{2g}(\FF_{\ell}) = ({\ell}^{2g}-1)({\ell}^{2g-2}-1) \cdots ({\ell}^2-1) {\ell}^{g^2}({\ell}-1).$$
We thus have
 $$r \leq [K:L'][L':L][L:\Q] \leq 2g e(g,p) [L:\Q]$$
where
$$e(g,p) = \begin{cases}
\# \GSp_{2g}(\FF_5), & p\neq 5 \\
\# \GSp_{2g}(\FF_7), & p \neq 7.
\end{cases} $$
This proves the following.
\begin{lem}\label{L:bound-on-frob}
		Let $A/L$ be a simple abelian variety of dimension $g$ with complex multiplication $j\colon  F \to \End(A/L')$ where $[F:\Q]=2g$. 
		Then $A$ has a lift of a $q$-Frobenius with $q \mid n(L,g,p)= p^{2g e(g,p)}$. 
		
		Letting $m(L,g,p)$ be the least common multiple of the numbers less than $n(L,g,p)$ we show that if $A$ does not have a lift of the $p^m$-Frobenius then $A$ does not have a lift of the $p^r$-Frobenius for $r<n$. In particular, $A/L$ does not have complex multiplication. 
\end{lem}	
This proves the Lemma \ref{L:CM-field} from the introduction.

\begin{rem}
The power of the lift of the Frobenius in this statement is unnecessarily large. 
In particular, the power $p^r$ is large enough so that the Frobenius power we are lifting acts linearly on the residue fields $\FF_q$. 
A more sophisticated approach to lifting the Frobenius has to do with the Serre tensor construction (\cite[1.7.4]{Chai2014}; see also \cite[Chapter 3, Section 2]{Lang2012}, where these are called $\mathfrak{a}$-transforms) but requires an additional hypothesis of $\End(A/k) = \OO_F$ (see \cite[Ch 3, Proposition 3.1]{Lang2012}).
If $\End(B/k) = R$ is an order in $\OO_F$ one has an isogeny $B \to \OO_F\otimes_R B$ of degree $[\OO_F:R]$. 
It is unclear to the authors at the time of writing this if this allows us to remove the dependence on the degree of the field of moduli in Lemma~\ref{L:bound-on-frob}.

Finally, for a CM abelian variety, its torsion field actually has \emph{abelian} Galois group; abelian subgroups of the general symplectic group $\GSp_{2g}(\FF_{\ell})$ have order at most $\ell^{g(2g+1)+1}$ \cite[Table 2]{Vdovin:maximal-orders-of-abelian}, while $\GSp_{2g}(\FF_\ell)$ itself has order roughly $\ell^{2^{g-1} \cdot (g+1)g + g^2 + 1}$, which gives a small improvement to the constant $n(L,g,p)$.

\end{rem}

%\newpage

%\includegraphics[width=.500\textwidth]{important-picture.png}

\bibliographystyle{alpha}
\bibliography{wt}

\vspace{.2in}

\end{document}